\documentclass{amsart}
\usepackage[a4paper, margin=3cm]{geometry}
\usepackage{graphicx} 
\usepackage{amsmath}
\usepackage{hyperref}

\def\Dom{\text{\normalfont Dom}}
\newcommand{\A}{\mathcal{L}_a^s}

\newcommand{\R}{{\mathbb R}}

\newcommand{\N}{{\mathbb N}}

\newcommand{\ve}{\varepsilon}
\newcommand{\EE}{{\mathcal A}^s}
\newcommand{\G}{{\mathcal G}}
\newcommand{\LL}{{\mathcal L}}

\newtheorem{teo}{Theorem}[section]
\newtheorem{lema}[teo]{Lemma}
\newtheorem{prop}[teo]{Proposition}

\theoremstyle{remark}

\theoremstyle{definition}
\newtheorem{defi}[teo]{Definition}

\newtheorem{conditions}[teo]{Conditions}
\def\diver{\mathop{\text{\normalfont div}}}

\numberwithin{equation}{section}
\begin{document}
\title[Nonlocal quasilinear anisotropic equation]{An eigenvalue problem for a nonlocal quasilinear anisotropic equation in fractional Orlicz Sobolev spaces without the $\Delta_2$--condition}
\author{Juli\'an Fern\'andez Bonder}

\address[JFB]{Departamento de Matem\'atica, FCEN, Universidad de Buenos Aires and \hfill\break\indent
Instituto de C\'alculo, CONICET \hfill\break\indent
Ciudad Universitaria, $0+\infty$ Building, Intendente Güiraldes 2160, 
Buenos Aires, Argentina}

\email{jfbonder@dm.uba.ar}

\urladdr{http://mate.dm.uba.ar/~jfbonder/}

\author{Mart\'in S. Guzm\'an}

\address[MG]{Departamento de Matem\'atica, FCFMyN, Universidad Nacional de San
Luis \hfill\break\indent Instituto de Matem\'atica Aplicada San
Luis, IMASL, CONICET. \hfill\break\indent Italia avenue 1556, office
5, San Luis (5700), San Luis, Argentina.}

\email{msguzman@unsl.edu.ar}

\author{Juan F. Spedaletti}

\address[JS]{Departamento de Matem\'atica, FCFMyN, Universidad Nacional de San
Luis \hfill\break\indent Instituto de Matem\'atica Aplicada San
Luis, IMASL, CONICET. \hfill\break\indent Italia avenue 1556, office
5, San Luis (5700), San Luis, Argentina.}

\email{jfspedaletti@unsl.edu.ar}


\begin{abstract}
In this paper we analyze an eigenvalue problem associated to fractional operators of the form
$$
\A u(x)=2\text{p.v.}\int_{\R^n}a(x,y,D^su(x,y))\,\frac{dy}{|x-y|^{n+s}},
$$
which represents a generalization model for nonlocal, nonstandard growth diffusion problems. We study this problem in the context of the fractional Orlicz Sobolev spaces without assuming the so-called $\Delta_2$--condition on the Young functions involved. We show existence of a sequence of eigenpairs $(u_k,\lambda_k)\to (0,+\infty)$.
\end{abstract}
\maketitle

\section{Introduction}
Eigenvalue problems are among the most important topics in Partial Differential Equations due to their wide range of applications in both pure and applied mathematics. See, for instance, the classical books by Courant and Hilbert \cite{CourantHilbertV1, CourantHilbertV2}.

In particular, for linear elliptic operators such as the Laplacian, the study of the asymptotic behavior of eigenvalues dates back to Courant. The classical Courant minimax principle ensures the existence of an infinite sequence of eigenvalues $\{\lambda_k\}_{k\in \N}$ for the Dirichlet eigenvalue problem
$$
\begin{cases}
-\Delta u = \lambda u & \text{in }\Omega\\
u=0 & \text{in } \partial\Omega,
\end{cases}
$$
where $\lambda_k \to \infty$ and $\Omega \subset \R^n$ is a bounded open set.

In the early 1990s, attention shifted to nonlinear extensions of eigenvalue problems associated with the Laplacian. One of the most extensively studied examples is the eigenvalue problem for the $p$-Laplacian:
\begin{equation}\label{p-lap-eigen}
\begin{cases}
-\Delta_p u = \lambda |u|^{p-2}u & \text{in }\Omega,\\
u=0 & \text{on } \partial\Omega,
\end{cases}
\end{equation}
introduced in \cite{L} (see also \cite{L2, L3}).

A standard approach to deal with nonlinear eigenvalue problems such as \eqref{p-lap-eigen} is based on the Ljusternik--Schnirelmann theory, introduced in the early 1930s by Ljusternik and Schnirelmann \cite{LS}. This theory provides a method for finding critical points of differentiable functionals on finite-dimensional Riemannian manifolds. Its extension to infinite-dimensional settings was developed by Schwartz \cite{Schwartz1,Schwartz} for Riemannian manifolds modeled on Hilbert spaces, and by Palais \cite{Palais} for Finsler manifolds modeled on Banach spaces.

These ideas have been widely used to establish the existence of infinitely many distinct solutions for nonlinear PDEs. A very incomplete list of contributions includes \cite{Berger, Browder1, Browder2, Browder3, Browder4, Coffman1, Coffman2, Coffman3, Krasnoselskii1956, Vainberg}.

To apply the Ljusternik--Schnirelmann theory to problem \eqref{p-lap-eigen}, one considers the functionals $F$ and $G$ defined on $W_0^{1,p}(\Omega)$ by
$$
F(u)=\int_\Omega |\nabla u|^p\,dx\qquad  \text{ and }\qquad G(u)=\int_\Omega|u|^p\,dx.
$$
In this framework, one obtains the existence of a sequence of eigenvalues $\{\lambda_k\}_{k\in\N}$ with $\lambda_k \to \infty$. A crucial point is that $W_0^{1,p}(\Omega)$ is reflexive and separable, and the functionals $F$ and $G$ are differentiable (see \cite{FuNe, FuNeSoSo, Zei80}).

Another interesting nonlinear eigenvalue problem arises with the $m-$Laplacian operator, defined by
$$
\Delta_m u = \diver\left(\frac{m(|\nabla u|)}{|\nabla u|}\nabla u\right),
$$
where $m\colon \R_+\to\R_+$ is a nondecreasing function. This operator generalizes the $p-$Laplacian operator when $m(t)=t^{p-1}$.

The corresponding eigenvalue problem reads
\begin{equation}\label{intro.1}
\begin{cases}
-\Delta_m u=\lambda g(u)& \text{in }\Omega\\
u=0 & \text{in }\partial\Omega.
\end{cases}
\end{equation}
Here, the function $g\colon \R \to \R$ satisfies suitable growth conditions. These operators are particularly appealing in applications due to the possibility of different diffusion behaviors when $|\nabla u|\ll 1$ and $|\nabla u|\gg 1$, a phenomenon commonly referred to as \emph{nonstandard growth}.

A key role is played by the primitive function of $m$, defined as $M(t)=\int_0^t m(s),ds$. When $M$ satisfies the $\Delta_2$-condition, namely
$$
M(2t)\leq C M(t),
$$
for some constant $C>1$ and all $t\geq T_0$, problem \eqref{intro.1} shares many properties with the $p$-Laplacian case. In particular, the Ljusternik--Schnirelmann theory can be applied to obtain a sequence of eigenvalues.

However, when $M$ does not satisfy the $\Delta_2$-condition, the situation becomes significantly more delicate. In this case, the associated functional spaces are no longer separable or reflexive, and the corresponding functionals fail to be differentiable. In \cite{Ti}, the author studied problem \eqref{intro.1} by combining ideas from \cite{Gossez} with a Galerkin-type approximation method, overcoming the lack of the $\Delta_2$-condition. As a consequence, the Ljusternik--Schnirelmann theory can still be applied, yielding the existence of an infinite sequence of eigenvalues for \eqref{intro.1}.

Recently, considerable attention has been devoted to the study of nonlocal diffusion problems due to their numerous applications in the natural sciences. See \cite{AkgirayBooth1988, Constantin2006, Eringen2002, GilboaOsher2008}. Arguably, one of the most relevant nonlocal operators is the \emph{fractional Laplacian}, defined by
$$
(-\Delta)^s u(x) = C(n, s)\text{p.v.}\int_{\R^n} \frac{u(x)-u(y)}{|x-y|^{n+2s}}\, dy.
$$
A natural question in this context is the study of the corresponding eigenvalue problem and its relationship with the classical Laplacian. This problem has been addressed by several authors, and a comprehensive treatment can be found in \cite{MolicaBisciRadulescuServadei2016}.

Nonlinear extensions of this problem were considered subsequently. In particular, in \cite{LindgrenLindqvist2014, FP} the authors study the eigenvalue problem associated with the fractional $p$-Laplacian, defined as
$$
(-\Delta_p)^s u(x) = C(n, s, p)\text{p.v.}\int_{\R^n} \frac{|u(x)-u(y)|^{p-2}(u(x)-u(y))}{|x-y|^{n+sp}}\, dy.
$$
In this setting, the classical Ljusternik--Schnirelmann theory can be applied, since the associated functional spaces, namely the fractional Sobolev spaces $W^{s,p}_0(\Omega)$, are reflexive and separable, and the corresponding functionals are differentiable.

More recently, in \cite{FBSP}, the authors analyzed an eigenvalue problem associated with the fractional $m$-Laplacian operator defined by
$$
(-\Delta_m)^s u(x)=  \text{p.v.}\int_{\R^n} m\left ( \frac{u(x)-u(y)}{|x-y|^s}\right )\,\frac{dxdy}{|x-y|^{n+s}},
$$ 
where $m=M'$ and $M$ is a Young function. Note that when $m(t)=|t|^{p-2}t$, this operator reduces to the fractional $p$-Laplacian.

The corresponding eigenvalue problem is
\begin{equation}\label{intro.2}
\begin{cases}
(-\Delta_m)^s u=\lambda g(u)& \text{in }\Omega\\
u=0 & \text{in }\R^n\setminus \Omega,
\end{cases}
\end{equation}
where the function $g\colon \R\to \R$ satisfies suitable growth conditions. In \cite{FBSP}, without assuming the $\Delta_2$-condition on either $M$ or its conjugate $\bar{M}$, the authors employ a Galerkin-type method to obtain a countable sequence of distinct eigensolutions ${(u_k,\lambda_k)}_{k\in\N}$ to problem \eqref{intro.2}. Moreover, it is shown that $\lambda_k \to +\infty$ as $k\to\infty$.

The aim of the present paper is to generalize the results in \cite{FBSP} by considering a class of anisotropic fractional operators $\A$ defined by
$$
\A u(x)=2\text{p.v.}\int_{\R^n} a\left (x,y,\frac{u(x)-u(y)}{|x-y|^s}\right)\,\frac{dy}{|x-y|^{n+s}},
$$
where the function $a$ is odd, symmetric, monotone, and satisfies suitable growth conditions (see Conditions \ref{condicionesa} for precise assumptions).

In this framework, we study the associated nonlocal eigenvalue problem
\begin{equation}\label{intro.3}
\begin{cases}
\A u=\lambda g(u)& \text{in }\Omega\\
u=0 & \text{in }\R^n\setminus \Omega,
\end{cases}
\end{equation}
where $g\colon \R\to\R$ satisfies appropriate growth conditions.

Our main result establishes the existence of a sequence of eigenvalues ${\lambda_k}_{k\in\N}$ for problem \eqref{intro.3}, such that $\lambda_k \to +\infty$ as $k\to\infty$ (see Theorem \ref{teo.main}).

\subsection*{Organization of the paper}
The paper is organized as follows: after this introduction, in Section \ref{prel} we include the preliminaries on Young functions, Orlicz spaces and fractional Orlicz-Sobolev spaces. Even though the content is not new, there is some not so standard notation and terminology that will be used throughout the paper. In Section \ref{sec:operators} we precisely define the operators $\A$ and prove some of the main properties of these operators, namely the pseudo-monotonicity (see Theorem \ref{pseudo}). Finally, Section \ref{eigenproblem} we analyze the eigenvalue problem \eqref{intro.3} and prove our main result about the existence of eigenpairs, Theorem \ref{teo.main}.

\section{Preliminaries}\label{prel}

In this section we present some preliminary definitions needed for the rest of the paper. The first subsection is well known and does not contain any new result the book \cite{KR} being the standard reference for the subject. The second subsection contains the definitions and basic results regarding fractional Orlicz-Sobolev spaces. See for instance \cite{FBS} where these spaces were introduced and \cite{ACPS2, ACPS} where several properties of these spaces were analyzed without requiring the $\Delta_2$--condition. The third subsection recalls the definition of complementary pairs introduced in \cite{Gossez} and the definition of the segment property, and construct a complementary pair in the context of fractional Orlicz-Sobolev spaces.  In the fourth subsection we recall an abstract result due to \cite{Ti} that will be helpful in the sequel. Finally, in the fifth subsection we state the definition of a deformation of a set that will play a key role in the regularity of the level sets that we will use in section \ref{eigenproblem}.  

\subsection{Young functions and Orlicz spaces}

Let $M\colon \R\to\R$ be an even, convex and continuous function, such that $M(t)>0$ for $t>0$, $M(t)/t\to 0$ as $t\to 0$ and $M(t)/t\to\infty$ as $t\to\infty$. Such a function $M$ is called a {\em Young function} if it can be written as
$$
M(t) = \int_0^{|t|} m(s)\, ds,
$$
for $m\colon [0,\infty)\to [0,\infty)$ increasing, right continuous with $m(t)=0$ if and only if $t=0$ and $m(t)\to\infty$ as $t\to\infty$.

It will be helpful to extend the function $m$ to the entire real line by oddness, that is $m(t) = \text{sign}(t) m(|t|)$.

We recall now some basic definitions on Orlicz spaces that can be found, for instance, in \cite{KR}.

Let $U\subset \R^N$ be a bounded domain and let $\mu$ be a Borel measure in $U$. The Orlicz class $\LL_M(U, d\mu)$ is defined as
$$
\LL_M(U,d\mu) := \left\{u\colon U\to \R,\ \text{measurable}\colon \int_U M(u)\, d\mu < \infty\right\}.
$$
The Orlicz space $L_M(U,d\mu)$ is then defined as the linear hull of $\LL_M(U,d\mu)$. It follows that $L_M(U,d\mu)$ can be characterized as
$$
L_M(U,d\mu) = \left\{ u\colon U\to \R,\ \text{$\mu-$measurable}\colon \int_U M\left (\frac{u}{k}\right )\, d\mu < \infty, \text{ for some } k>0\right\}.
$$
This space is a Banach space when it is equipped, for instance, with the {\em Luxemburg norm}, i.e.
$$
\|u\|_{L_M(U,d\mu)} = \|u\|_{M,U,d\mu} = \|u\|_{M,d\mu} := \inf\left\{ k>0\colon \int_U M\left(\frac{u}{k}\right)\, d\mu \leq 1\right\}.
$$
A well-known and interesting fact is that $\LL_M(U,d\mu) = L_M(U,d\mu)$ if and only if $M$ satisfies the so-called $\Delta_2$--condition, i.e.
\begin{equation}\label{Delta2}
M(2t)\le CM(t),\quad \text{for } t\ge T.
\end{equation}
Also, the Orlicz space $L_M(U,d\mu)$ is separable if and only if $M$ satisfies \eqref{Delta2}.

Next, we define the space $E_M(U,d\mu)$ as the closure of bounded $\mu-$measurable functions in $L_M(U,d\mu)$, in the case $\mu(U)=\infty$ the space $E_M(U,d\mu)$ is the closure in $L_M(U,d\mu)$ of bounded $\mu-$measurable functions with bounded support. It is easy to see that
$$
E_M(U,d\mu) = \left\{ u\colon U\to \R,\ \text{$\mu-$measurable}\colon \int_U M\left (\frac{u}{k}\right )\, d\mu < \infty, \text{ for every } k>0\right\}.
$$
Again, $E_M(U,d\mu)=L_M(U,d\mu)$ if and only if $M$ satisfies \eqref{Delta2}.

So, in general, we have
$$
E_M(U,d\mu)\subset \LL_M(U,d\mu)\subset L_M(U,d\mu),
$$
with equalities if and only if $M$ satisfies \eqref{Delta2}.

Observe that $E_M(U,d\mu)$ and $L_M(U,d\mu)$ are Banach spaces and $\LL_M(U,d\mu)$ is a convex set.

\bigskip

Given a Young function $M$, we define its complementary function $\bar M$ as
$$
\bar M(t) := \sup\{ \tau|t|-M(\tau)\colon \tau\geq 0\}.
$$
Observe that $\bar M$ is also a Young function and it is the optimal function in the Young inequality
\begin{equation}\label{young}
\tau t\le M(t) + \bar M(\tau),
\end{equation}
for all $\tau,t\in \R$ . Observe that equality in \eqref{young} is achieved if and only if $\tau=\text{sign}(t)m(t)$ or $t=\text{sign}(\tau)\bar m(\tau)$ where $\bar{m}(t)$ is the derivative of $\bar{M}(t)$. 

It follows directly from \eqref{young} that if $u\in L_M(U,d\mu)$ and $v\in L_{\bar M}(U,d\mu)$, then $uv\in L^1(U)$ and
$$
\int_U |uv|\, d\mu \le 2 \|u\|_M \|v\|_{\bar M}.
$$
This fact allows one to define in $L_M(U,d\mu)$ the topology $\sigma(L_M,L_{\bar M})$ and it follows that $E_M(U,d\mu)$ is dense in $L_M(U,d\mu)$ in this topology.

It is easy to check that $\bar{\bar M} = M$. The Orlicz space $L_{\bar M}(U,d\mu)$ is the dual space of $E_M(U,d\mu)$ and so $L_M(U,d\mu)$ is reflexive if and only if $M$ and $\bar M$ satisfy \eqref{Delta2}.

Finally, given $M$ a Young function, we define
\begin{equation}\label{Domm}
\Dom(m) := \{u\in L_M(U,d\mu)\colon m(|u|)\in L_{\bar M}(U,d\mu)\}.
\end{equation}
It can be checked that $E_M(U,d\mu)\subset \Dom(m)\subset \LL_M(U,d\mu)$ and hence, $\Dom(m)=L_M(U,d\mu)$ if and only if $M$ satisfies \eqref{Delta2}. Moreover, the map $u\mapsto m(u)$ from $E_M(U,d\mu)$ to $L_{\bar M}(U,d\mu)$ is continuous if and only if $\bar M$ satisfies \eqref{Delta2}.

\subsection{Fractional order Orlicz-Sobolev spaces}
In the product space $\R^n\times\R^n = \R^{2n}$ we define the measure
$$
d\nu_n := \frac{dxdy}{|x-y|^n}.
$$
Observe that this is a Borel measure and that if $K\subset \R^{2n}\setminus \Delta$ is compact, then $\nu_n(K)<\infty$, where $\Delta\subset \R^n\times\R^n$ is the diagonal $
\Delta := \{(x,x)\colon x\in\R^n\}$.

We will consider two Orlicz spaces $L_M(U,d\mu)$. One with $U=\R^n$ and $d\mu=dx$ (the Lebesgue measure) and another with $U=\R^{2n}$ and $d\mu = d\nu_n$. 

We will use the notations
$$
\LL_M = \LL_M(\R^n, dx), \qquad L_M = L_M(\R^n, dx),\qquad E_M = E_M(\R^n, dx);
$$
$$
\LL_M(\nu_n) = \LL_M(\R^{2n}, d\nu_n), \quad L_M(\nu_n) = L_M(\R^{2n}, d\nu_n),\quad E_M(\nu_n) = E_M(\R^{2n}, d\nu_n).
$$
Now, given a fractional parameter $s\in (0,1)$, we introduce the notation for the $s-$H\"older quotient of a function $u\colon\R^n\to \R$.
$$
D^su(x,y) := \frac{u(x)-u(y)}{|x-y|^s}.
$$
Then $D^su\colon \R^{2n}\setminus \Delta\to \R$.

Now, with all the notation introduced, the fractional order Orlicz-Sobolev spaces are defined as
$$
W^sL_M := \{u\in L_M\colon D^s u\in L_M(\nu_n)\}
$$
and
$$
W^sE_M :=\{u\in E_M\colon D^su\in E_M(\nu_n)\}.
$$
These spaces are naturally equipped with the norms
$$
\|u\|_{s,M} = \|u\|_M + \|D^su\|_{M,\nu_n}.
$$

Also, these spaces can be isometrically identified as closed subspaces of $L_M\times L_M(\nu_n)$ and $E_M\times E_M(\nu_n)$ respectively using the map
$$
u\mapsto (u, D^su).
$$
Now, given $\Omega\subset\R^n$ a bounded open set, the space $W^s_0L_M(\Omega)$ is then defined as the closure of ${\mathcal D}(\Omega)$ in $W^sL_M$ with respect to the topology $\sigma(L_M\times L_M(\nu_n), E_{\bar M}\times E_{\bar M}(\nu_n))$. By Poincar\'e's inequality (see \cite[Corollary 6.2]{FBS}) in  $W_0^s L_M(\Omega)$ the quantity $\|D^su\|_{M,\nu_n}$ is equivalent to $\|u\|_{s, M}$. Therefore in $W_0^s L_M(\Omega)$ we will consider the norm $\|u\|_{W^s_0L_M(\Omega)} = \|D^su\|_{M,\nu_n}$.

The space $W^s_0 E_M(\Omega)$ is defined as the closure of ${\mathcal D}(\Omega)$ in $W^sE_M$ in norm topology.

In order to define the dual spaces, we need to introduce the notion of fractional divergence. See \cite{FBSaRi}.

Given $F\in L_{\bar{M}}(\nu_n)$, the fractional divergence of $F$ is defined as
\begin{align*}
{\mathrm{div}}^s F(x) &:= \text{p.v.} \int_{\R^n} \frac{F(y,x)-F(x,y)}{|x-y|^{n+s}}\, dy \\
&= \lim_{\ve\to 0} \int_{\R^n\setminus B_\ve(x)}\frac{F(y,x)-F(x,y)}{|x-y|^{n+s}}\, dy.
\end{align*}
In \cite{FBPLS} it is shown that if $F\in L_{\bar M}(\nu_n)$, then ${\mathrm{div}}^s F\in (W^s_0L_M(\Omega))^*$ and the following {\em fractional integration by parts formula} holds
\begin{equation}\label{int.by.parts}
\langle {\mathrm{div}}^s F, u\rangle = -\iint_{\R^{2n}} F D^s u\, d\nu_n.
\end{equation}
So, we define the following spaces of distributions
$$
W^{-s}L_{\bar M}(\Omega) := \{\phi\in {\mathcal D}'(\Omega)\colon \phi = f + \mathrm{div}^sF \text{ with } f\in L_{\bar M},\ F\in L_{\bar M}(\nu_n)\}
$$
$$
W^{-s}E_{\bar M}(\Omega) := \{\phi\in {\mathcal D}'(\Omega)\colon \phi = f + \mathrm{div}^sF \text{ with } f\in E_{\bar M},\ F\in E_{\bar M}(\nu_n)\}.
$$
Recall that since $E_{\bar{M}}$ and $E_{\bar{M}}(\nu_n)$ are separable then $W^{-s}E_{\bar M}(\Omega)$ is also separable.

These spaces are endowed with the usual quotient norms,
$$
\|\phi\|_{-s, \bar M} := \inf\{\|f\|_{\bar M} + \|F\|_{\bar M, \nu_n}\colon \phi=f+\mathrm{div}^s F\}.
$$

\subsection{Complementary systems and segment property}\label{complementary}

In \cite{Do, DoTru} the authors introduce the notion of complementary systems in order to work in spaces without the usual reflexivity assumption.

Let $Y$ and $Z$ be real Banach spaces with a pairing $\langle\cdot,\cdot\rangle\colon Y\times Z\to \R$. Let $Y_0\subset Y$ and $Z_0\subset Z$ be closed and separable subspaces. We say $(Y, Y_0; Z, Z_0)$ is a complementary system if $Y_0^*=Z$ and $Z_0^*=Y$ (where equality is understood in the sense of a natural isometry via the pairing $\langle\cdot,\cdot\rangle$).

The first natural example of a complementary system is $Y=L_M$, $Y_0=E_M$, $Z=L_{\bar M}$ and $Z_0=E_{\bar M}$.

Observe that it is immediate to see that
\begin{equation}\label{LMEM}
\left (L_M\times L_M(\nu_n),E_M\times E_M(\nu_n); L_{\bar M}\times L_{\bar M}(\nu_n),E_{\bar M}\times E_{\bar M}(\nu_n)
\right)
\end{equation}
is also a complementary system.

In \cite{Gossez} the author provides with a general method to generate complementary systems from a previous one. More precisely
\begin{lema}
	[\cite{Gossez}, Lemma 1.2]\label{lemaGossez}
Given a complementary system $(Y, Y_0; Z, Z_0)$ and a closed subspace $E$ of $Y$, define $E_0= E \cap Y_0$, $F = Z/E_0^\perp$ and $F_0 = Z_0/E_0^\perp$.

Then, the pairing $\langle\cdot,\cdot\rangle$ between $Y$ and $Z$ induces a pairing between $E$ and $F$ if and only if $E_0$ is $\sigma(Y,Z)$ dense in $E$ that is if $u\in E$ then there exists $\{u_m\}_{m\in\N}\in E_0$ such that $\langle u_m,v\rangle\to \langle u,v\rangle$ if $m\to \infty, \forall v\in Z$. In this case, $(E, E_0; F, F_0)$ is a complementary system if $E$ is $\sigma(Y, Z_0)$ closed, and conversely, when $Z_0$ is complete, $E$ is $\sigma(Y, Z_0)$ closed if $(E, E_0; F, F_0)$ is a complementary system.
\end{lema}

Using this Lemma, in \cite{Gossez} it is shown that 
$$
(W^1_0L_M(\Omega),W^1_0E_M(\Omega); W^{-1}L_{\bar M}(\Omega), W^{-1}E_{\bar M}(\Omega))
$$
is a complementary system when the domain $\Omega$ satisfies the {\em segment property}.

See Section 8.1 in \cite{DB} for the definition of the segment property.

Using similar ideas, in \cite[Section 2.3]{FBSP}  it is shown that if $\Omega$ satisfies the segment property, then
\begin{equation}\label{CS}
(W^s_0L_M(\Omega),W^s_0E_M(\Omega); W^{-s}L_{\bar M}(\Omega), W^{-s}E_{\bar M}(\Omega))
\end{equation}
is also a complementary system.

\section{Nonlocal, quasilinear, anisotropic operators} \label{sec:operators}
In this section, we will define the operators that will be the subject of our investigation. These operators are natural extensions of the fractional order $m-$Laplace operator that were introduced in \cite{FBS}. See also \cite{FBSP}.

To this end, let $a\colon \R^n\times\R^n\times \R\to \R$, $a=a(x,y,\xi)$. On this function $a$ we will impose the following conditions:

\begin{conditions}\label{condicionesa}
We will assume from now on that $a\colon \R^n\times\R^n\times \R\to \R$ satisfies
\begin{itemize}
\item{\em (Caratheodory condition)} $a(\cdot, \cdot, \xi)$ is a measurable for all $\xi\in\R$ and $a(x, y, \cdot)$ is a continuous function for a. e. $(x, y)\in \R^{2n}$.
\item $a(x,y,-\xi)=-a(x,y,\xi),$ for all $\xi\in\R$ and a.e. $(x, y)\in\R^{2n}$.
\item $a(x,y,\xi)\xi>0$ for a.e. $(x, y)\in\R^{2n}$ and all $\xi\in\R$ with $\xi\neq 0.$
\item {\em (Growth condition)} There exists a Young function $M$, a function $d\in E_{\bar{M}}(\nu_n)$ and constants $b,c\in \R^+$ such that  for a.e. $(x, y)\in \R^{2n}$ and all $\xi\in \R$
\begin{equation}\label{GC}
|a(x,y,\xi)|\leq d(x,y)+b\bar{M}^{-1}(M(c\xi))
\end{equation} 
where $\bar{M}$ is the complementary function of the Young function $M$.
\item{\em (Monotonicity condition)} For a.e. $(x, y)\in \R^{2n}$ and all $\xi,\xi'\in\R$ 
\begin{equation}\label{MC}
(a(x,y,\xi)-a(x,y,\xi'))(\xi-\xi')\geq 0.
\end{equation}
\end{itemize}
\end{conditions}

Given a function $a\colon \R^n\times\R^n\times \R\to \R$ that satisfies conditions \ref{condicionesa}, we define the operator $\A\colon \Dom(\A)\to W^{-s}L_{\bar{M}}$ as	
$$
\A u(x):=-\mathrm{div}^s(a(x,y,D^su(x,y))),
$$
where $\Dom(\A)=\{u\in W_0^sL_M(\Omega)\colon a(x, y, D^su(x, y))\in L_{\bar{M}}(\nu_n)\}$.

Observe that
$$
\begin{aligned}
\A u(x) &= -\mathrm{div}^s(a(x,y,D^su(x,y)))
\\
&=-\text{p.v.}\int_{\R^n}(a(y,x,D^su(y,x))-a(x,y,D^su(x,y)))\,\frac{dy}{|x-y|^{n+s}}\\
&=2\text{p.v.}\int_{\R^n}a_{\text{sym}}(x,y,D^su(x,y))\,\frac{dy}{|x-y|^{n+s}}
\\
&= 2\lim_{\varepsilon\downarrow 0}\int_{|x-y|\geq \varepsilon}a_{\text{sym}}(x,y,D^su(x,y))\,\frac{dy}{|x-y|^{n+s}},
\end{aligned}
$$
where $a_{\text{sym}}$ is the symmetric part of $a$ given by
$$
a_{\text{sym}}(x,y,\xi)=\frac{a(x, y, \xi)+a(y, x, \xi)}{2}.
$$

Observe that $a_{\text{sym}}$ also verifies conditions \ref{condicionesa} whenever $a$ does. So, without loss of generality, we may assume that $a$ is symmetric, i.e. $a=a_{\text{sym}}$, in this way
\begin{equation}\label{operador}
\A u(x)=2\lim_{\varepsilon\downarrow 0}\int_{|x-y|\geq \varepsilon}a(x,y,D^su(x,y))\,\frac{dy}{|x-y|^{n+s}},	
\end{equation}
where the kernel $a$ satisfies conditions \ref{condicionesa} and is symmetric.

We observe that, by definition, $\Dom(\A)\subset W_0^sL_M(\Omega)$. Let us now see that $W_0^sE_M(\Omega)\subset \Dom(\A)$.
\begin{lema}\label{WSEINDOM}
Let $0<s<1$ and $\Omega\subset \R^n$ be a bounded open set and let $a$ satisfies conditions \ref{condicionesa}. Then, if $u\in W_0^sE_M(\Omega)$ we have that $a(x,y,D^su)\in L_{\bar{M}}(\nu_n)$.
\end{lema}

\begin{proof}
Let $u\in W_0^sE_M(\Omega)$, then 
\begin{equation}\label{LIAEM1}
\iint_{\R^{2n}}M\left (|cD^s u| \right )\,d\nu_n<\infty,
\end{equation}	
for every constant $c\in\R$.

We claim that $b \bar{M}^{-1}(M(c|D^s u|))\in L_{\bar{M}}(\nu_n)$. Indeed by \eqref{LIAEM1}
\begin{align*}
\iint_{\R^{2n}}\bar{M}\left (\frac{b \bar{M}^{-1}(M(c|D^s u|))}{b} \right )\,d\nu_n =\iint_{\R^{2n}} M(c|D^su|)\,d\nu_n<\infty.
\end{align*}
The proof concludes thanks to the growth condition \eqref{GC}.
\end{proof}

Next, from the properties of the fractional divergence $\text{div}^s$, we get the following result
\begin{teo}
Let $0<s<1$ be fixed and let $a$ satisfies conditions \ref{condicionesa}. Then for every $u\in \Dom(\A)$, we have that $\A u\in W^{-s}L_{\bar{M}}(\Omega)$. Moreover, for every $v\in W_0^s L_{M}(\Omega)$,
\[
\langle \A u,v \rangle =\iint_{\R^{2n}}a(x,y,D^su)D^sv\,d\nu_n,
\]
\end{teo}

\begin{proof}
Observe that if we denote $F(x,y) = a(x, y, D^s u(x, y))$, if $u\in \Dom(\A)$ then $F\in L_{\bar M}(\nu_n)$. Therefore, the integration by parts formula \eqref{int.by.parts}, gives us
$$
\langle \A u,v \rangle = \langle-\text{div}^s F, v\rangle = \iint_{\R^{2n}} FD^sv\, d\nu_n = \iint_{\R^{2n}}a(x,y,D^su)D^sv\,d\nu_n,
$$
and the proof is complete.
\end{proof}

A very important property of the operator $\A$ is the pseudomonotonicity. The following result is the main theorem of this section.
\begin{teo}\label{pseudo}
Let $\Omega\subset \R^n$ be a bounded and open domain that satisfies the segment property. Then, the operator $\A$ is pseudomonotone. That is, If $\{u_i\}_{i\in \N}\subset \Dom(\A)$ is a sequence such that it fulfills
\begin{equation}\label{seq.pseudo}
\begin{cases}	
&u_i\to u\text{ for }\sigma(W_0^sL_M(\Omega),W^{-s}E_{\bar{M}}(\Omega))\\
&\A u_i\to f\in W^{-s}L_{\bar{M}}(\Omega)\text{ for }\sigma(W^{-s}L_{\bar{M}}(\Omega),W_0^sE_M(\Omega))\\
&\limsup_{i\to \infty}\langle \A u_i,u_i\rangle\leq \langle f,u\rangle
\end{cases}
\end{equation}
then 
\begin{equation}\label{tesis.pseudo}
\begin{cases}
&u\in \Dom(\A)
\\
&\A u=f
\\
&\langle \A u_i,u_i\rangle\to  \langle f,u\rangle\text{ if }i\to \infty.
\end{cases}
\end{equation}
\end{teo}

The proof of Theorem \ref{pseudo} is divided into a series of lemmas for the reader's convenience. It will be helpful to introduce the following notation: for a given function $u$ we define the sets
$$
R_j(u):=\{(x,y)\in \R^{2n}\colon |(x,y)|\leq j \text{ and }|D^s u|\leq j \}.
$$

\begin{lema}\label{paralimite}
Let $u \in W_0^s L_M(\Omega) \text{ and } v \in W_0^s E_M(\Omega) $ respectevely. If $0<|\lambda|<1$ then, for each $j \in \N$
$$
a(x,y,D^su+\lambda D^s v)D^s v,\quad a(x,y,D^s u)D^s v\in L^1(R_j(u),d\nu_n).
$$
Moreover
$$
\lim_{\lambda \to 0}\iint_{R_j(u)} a(x,y,D^su+\lambda D^s v)D^sv \,d\nu_n=\iint_{R_j(u)} a(x,y,D^s u)D^s v\,d\nu_n. 
$$
\end{lema}

\begin{proof}
Let $u\in W_0^sL_M(\Omega)$  and $v\in W_0^s E_M(\Omega)$ then using Young's inequality we have
\begin{align*}
|a(x,y,D^su+\lambda D^s v)D^s v| &= \left |\frac{a(x,y,D^su+\lambda D^s v)}{2 b}\right | \left |2  b D^s v\right | 
\\
&\le M(2 b D^s v) + \bar M\left (\frac{a(x,y,D^su+\lambda D^s v)}{2 b}\right ),	
\end{align*}
where $b\in\R$ is given by the growth condition \eqref{GC}.

Since $v\in W^s_0 E_M(\Omega)$ it follows that $M(2 b D^s v)\in L^1(\R^{2n}, d\nu_n)$ and by growth condition we have in $R_j(u)$
\begin{align*}
\bar M\left (\frac{a(x,y,D^su+\lambda D^s v)}{2 b}\right ) &\le \bar M\left (\frac{ d(x,y)}{2 b} +\frac{\bar M^{-1}(M(c (D^su+\lambda D^s v))}{2}\right )\\
&\le \frac12\left(\bar M\left (\frac{d(x, y)}{b}\right)+M(c(|D^s u|+|D^s v|)\right)
\\
&\le \frac{1}{2}\left ( \bar{M}\left (\frac{d(x, y)}{b}\right )+M(cj+c|D^s v|)\right )
\\
&\le \frac12 \left ( \bar{M}\left (\frac{d(x, y)}{b}\right )+\frac12 M(2cj)+\frac12 M(2c|D^s v|) \right ),
\end{align*}
using the fact that $d\in E_{\bar{M}}(\nu_n)$ and $v\in W_0^s E_M(\Omega)$ the right hand side in the last inequality is integrable in $R_j(u)$. We observe that the last inequality is valid uniformly in $\lambda$.

Finally, using the fact 
$$
a(x,y,D^su+\lambda D^sv)D^sv\to a(x,y,D^su)D^sv,
$$
as $\lambda\to 0$ a.e. in $R_j(u)$ and the dominated convergence theorem we conclude the proof.
\end{proof}
\begin{lema}\label{paraEulerLagrange}
If there exists $u \in W_0^sL_M(\Omega)$ and $\Phi \in L_{\bar{M}}(\nu_n)$ such that
\begin{equation}\label{desi}
\iint_{\R^{2n}} (a(x,y,W)-\Phi)(W-D^s u)\,d\nu_n\geq 0,	
\end{equation}
for all $W\in L^{\infty}(\R^{2n},d\nu_n)$ with compact support then $a(x,y,D^su)=\Phi \in (W_0^sL_M(\Omega))^*$ that is
$$
\iint_{\R^{2n}}a(x,y,D^su)D^sv\,d\nu_n =\iint_{\R^{2n}}\Phi D^sv\,d\nu_n \qquad \forall v \in W_0^sL_M(\Omega).
$$
\end{lema}
\begin{proof}
Let $w \in W_0^sL_M(\Omega)$ be such that $D^s w \in L^\infty(R_l(u),d\nu_n)$. Taking $l\geq j$ we define
$$
W:=D^s w\chi_{R_j(u)}-D^s u\chi_{R_j(u)}+D^s u\chi_{R_{l}(u)}=D^{s,j}w-D^{s,j}u+D^{s,l}u,
$$
then using $W$ in \eqref{desi}
\begin{equation}\label{desiconW}
\iint_{\mathbb{R}^{2n}}(a(x,y,D^{s,j}w-D^{s,j}u+D^{s,l}u)-\Phi)(D^{s,j}w-D^{s,j}u+D^{s,l}u-D^su)\,d\nu_n \geq 0.
\end{equation}
The left hand side on \eqref{desiconW} can be expressed as
\begin{align*}
&\iint_{\R^{2n}}(a(x,y,D^{s,j} w - D^{s,j} u + D^{s,l} u) - \Phi)(D^{s,j} w - D^{s,j} u)\,d\nu_n 
\\
&+\iint_{\R^{2n}}a(x,y,D^{s,j} w - D^{s,j} u + D^{s,l} u)(D^{s,l} u - D^s u)\,d\nu_n 
\\
&-\iint_{\R^{2n}}\Phi(D^{s,l} u-D^s u)\, d\nu_n
\\
&=I+II+III.	
\end{align*}
The first integral $I$ is zero outside $R_j(u)$, the second integral $II$  is zero, and $III$ integral tends to zero if $l\to \infty$. Hence, letting $l \to \infty$, we obtain
$$
\iint_{R_j(u)} \left(a(x,y,D^s w) - \Phi \right)(D^s w - D^s u)\, d\nu_n \geq 0.
$$
Now let $v \in \mathcal{D}(\Omega)$, using Lemma \ref{paralimite} with $\lambda > 0$, first with $w = u + \lambda v$ and then $w = u - \lambda v$, we have
$$
\iint_{R_j(u)} \left(a(x,y,D^s u) - \Phi \right) D^s v \, d\nu_n = 0 \quad \forall v \in \mathcal{D}(\Omega).
$$
Taking the limit $j \to \infty$, we have
$$
\iint_{\mathbb{R}^{2n}} \left(a(x,y,D^s u) - \Phi \right) D^s v \, d\nu_n = 0 \quad \forall v \in \mathcal{D}(\Omega),
$$
and by the density, $\forall v \in W_0^s E_M(\Omega)$. Using the density of $W_0^s E_M(\Omega) \subset W_0^s L_M(\Omega)$ with respect to the $\sigma(L_M \times L_M(\nu_n), L_{\bar{M}} \times L_{\bar{M}}(\nu_n))$ topology, we conclude the proof.
\end{proof}

Now we are in position to prove the pseudomonotonicity.
\begin{proof}[Proof of Theorem \ref{pseudo}]
Let $\{u_i\}_{i\in \N} \subset \Dom(\A)$ and $f\in  W^{-s}L_{\bar{M}}(\Omega)$ such that \eqref{seq.pseudo} is satisfied. 
We must prove that $u=\lim u_i$ and $f$ satisfy \eqref{tesis.pseudo}.

Let us first prove that the sequence $\{a(x,y,D^s u_i)\}_{i \in \N}$ remains bounded in $L_{\bar{M}}(\nu_n)$, indeed let $V \in L^{\infty}(\R^{2n},d\nu_n)$ with compact support, by the monotonicity condition \eqref{MC} we have 
\begin{equation}\label{mono}
\iint_{\R^{2n}} (a(x,y,D^s u_i)-a(x,y,V))(D^s u_i-V)\, d\nu_n\geq 0,
\end{equation}
which implies
\begin{align*}
\iint_{\R^{2n}} a(x,y,D^s u_i)V\,d\nu_n&\leq \iint_{\R^{2n}} a(x,y,D^s u_i)D^s u_i\,d\nu_n
\\
&-\iint_{\R^{2n}} a(x,y,V)D^s u_i\,d\nu_n
\\
&+\iint_{\R^{2n}} a(x,y,V)V\,d\nu_n
\\
&=I+II+III,
\end{align*}
the first integral remains bounded by \eqref{seq.pseudo}, the convergence $u_i\to u$ implies that the second integral is uniformly bounded in $i$ and the last integral is independent of $i$, therefore there exists a constant $C$, independent of $i$, such that
$$
\iint_{\R^{2n}} a(x,y,D^s u_i)V\,d\nu_n\leq C,
$$
for all $V\in L^{\infty}(\R^{2n},d\nu_n)$ with compact support and therefore for all $V\in E_M(\nu_n)$, then by the uniform boundedness principle we have $\|a(x,y,D^s u_i)\|_{L_{\bar{M}}(\nu_n)}\leq C$ uniformly in $i$ and so there exists a subsequence that we still denote by $\{a(x,y,D^s u_i)\}_{i\in \N}$ and   $\Phi\in L_{\bar{M}}(\nu_n)$ such that $a(x,y,D^su_i)\to \Phi$ in $\sigma(L_{\bar{M}}(\nu_n),E_M(\nu_n))$, using the above convergence and the fact $\A u_i \to f \in W^{-s}L_{\bar{M}}(\Omega)$ for $\sigma(W^{-s}L_{\bar{M}}(\Omega),W_0^sE_M(\Omega))$ we have
\begin{align*}
\langle f,v \rangle
&=\lim_{i\to\infty} \langle \A u_i,v\rangle
\\
&=\lim_{i\to\infty}\iint_{\R^{2n}}a(x,y,D^s u_i)D^sv\,d\nu_n
\\
&=\iint_{\R^{2n}}\Phi D^s v\,d\nu_n,
\end{align*}
for all $v\in W_0^sE_M(\Omega)$. This formula together with the density of $W_0^s
E_M(\Omega)$ in $W_0^s
L_M(\Omega)$ with respect to the $\sigma(L_M\times L_M(\nu_n),L_{\bar{M}}\times L_{\bar{M}}(\nu_n)$) topology allows us to extend $v$ to the space $W_0^sL_M(\Omega)$, that is
\begin{equation}\label{deff}
\langle f,v \rangle=\iint_{\R^{2n}}\Phi D^s v\,d\nu_n\qquad \forall v\in W_0^s L_M(\Omega).	
\end{equation}
Now taking limit in \eqref{mono} we obtain
$$
\iint_{\R^{2n}} (\Phi-a(x,y,V))(D^s u-V)\, d\nu_n\geq 0,
$$
for all $V\in L^{\infty}(\R^{2n},d\nu_n)$ with compact support and so by Lemma \ref{paraEulerLagrange} we have $a(x,y,D^s u)=\Phi$ in $(W_0^sL_M(\Omega))^*$, that is $\A u=f$.

On the other hand putting $u_i=u$ in \eqref{mono}, using $\A u=f$ and \eqref{deff} we have
\begin{align*}
\iint_{\R^{2n}} a(x,y,D^s u)V\,d\nu_n&\leq \iint_{\R^{2n}} \Phi D^s u\,d\nu_n
\\
&-\iint_{\R^{2n}} a(x,y,V)D^s u\,d\nu_n
\\
&+\iint_{\R^{2n}} a(x,y,V)V\,d\nu_n,
\\
&<\infty,
\end{align*}
for all $V\in E_M(\nu_n)$ that is $a(x,y,D^su)\in L_{\bar{M}}(\nu_n)$ and so $u\in \Dom(\A)$.

Finally, we prove that 
$$
\langle \A u_i, u_i \rangle \to \langle f, u \rangle \quad \text{for } i \to \infty.
$$
Let  
$$
L = \liminf_{i \to \infty} \iint_{\mathbb{R}^{2n}} a(x,y,D^s u_i) D^s u_i \, d\nu_n,
$$
we only need to prove that \( L \geq \langle f, u \rangle \). 

Using the monotonicity property again we have  
$$
\iint_{\mathbb{R}^{2n}} \left(a(x,y,D^s u_i) - a(x,y,D^{s,j}u) \right) (D^s u_i - D^{s,j}u) \, d\nu_n \geq 0,
$$
where $ D^{s,j}u = D^s u \chi_{R_j(u)}$. Taking the limit in \( i \) and rewriting, we obtain  
$$
L \geq \iint_{\mathbb{R}^{2n}} a(x,y,D^{s,j}u)(D^s u - D^{s,j}u) \, d\nu_n + \iint_{\mathbb{R}^{2n}} a(x,y,D^s u) D^{s,j}u \, d\nu_n.
$$
In $ R_j(u)$ the factor $D^s u - D^{s,j}u = 0$, and in $R_j(u)^c$ the factor $a(x,y,D^{s,j}u)=a(x,y,0)=0$, so the first integral in the above inequality is zero. Therefore  
$$
L \geq \iint_{R_j(u)}a(x,y,D^s u) D^s u \, d\nu_n,
$$
for arbitrary $j$, then  
$$
L \geq \iint_{\mathbb{R}^{2n}}a(x,y,D^s u) D^s u \, d\nu_n = \langle f, u \rangle.
$$
This concludes the proof of the theorem.
\end{proof}
\section{The eigenvalue problem}\label{eigenproblem}
In this section we analyze the eigenvalue problem
\begin{equation}\label{mainproblem}
\begin{cases}
\A u=\lambda g(x, u) &\text{in }\Omega\\
u=0&\text{in }\R^n\setminus \Omega,
\end{cases}
\end{equation}
where $a(x, y, \xi)$ verifies conditions \ref{condicionesa} and, moreover, we also assume the {\em coercitivity condition}
\begin{equation}\label{coersive}
a(x, y, \xi)\xi\ge \theta M(c\xi), \quad \text{for every $\xi\in \R$ and a.e. $(x, y)\in \R^{2n},$}
\end{equation}
for some constants $\theta, c>0$.

On the source term $g\colon \R\to \R$ we will assume that it is an odd and continuous function such that $g(t)>0$ for all $t> 0$ and that it verifies the growth condition
\begin{equation}\label{condg}
|g(x, t)|\leq e(x) + a_1 \bar M^{-1}(a_2 M(a_3 t))\qquad \forall t\geq 0,
\end{equation}
where $a_1,a_2$ and $a_3$ are positive constants and $e(x)\in L_{\bar M}$. 

Next, we define the notion of eigenpair for \eqref{mainproblem}:
\begin{defi}
We say that $(\lambda,u)$ is an eigenpair for \eqref{mainproblem} if $u\in \Dom(\A)$ and
\begin{equation}\label{solmain}
\iint_{\R^{2n}}a(x,y,D^su)D^s v\,d\nu_n=\lambda \int_\Omega g(u)v\,dx,
\end{equation} 
for all $v\in W_0^s L_M(\Omega)$. 
\end{defi}

Recall that $\Dom(\A)\subset W^s_0L_M(\Omega)\subset E_M(\Omega)$ and therefore the right hand side of \eqref{solmain} is well defined. In fact,
$$
\int_\Omega |g(u)||v|\,dx\le \int_\Omega e(x)|v|\, dx + a_1\int_\Omega \bar M^{-1}(a_2 M(a_3 |u|))|v|\, dx.
$$
The first integral is well defined by Young's inequality \eqref{young} and the second one, using again Young's inequality it is bounded by
$$
ka_2\int_\Omega M(a_3 |u|)\, dx + k\int_\Omega M\left(\frac{|v|}{k}\right)\, dx <\infty. 
$$

In order to study problem \eqref{mainproblem} we define the functionals 
$$
\EE \colon \Dom(\EE)\subset W^s_0L_M(\Omega)\to \R \quad \text{and}\quad \mathcal{G}\colon W^s_0L_M(\Omega)\to \R
$$ 
as 
$$
\EE(u):=\iint_{\R^{2n}}A(x,y,D^s u)\,d\nu_n\quad \text{and}\quad \mathcal{G}(u):=\int_\Omega G(u)\,dx,
$$
respectively, where $A(x,y,t)=\int_0^t a(x,y,\tau)\, d\tau, G(t)=\int_0^t g(\tau)\, d\tau$ and
\begin{align*}
&\Dom(\EE):=\{u\in W_0^sL_M(\Omega)\colon  \EE(u)<\infty\}.
\end{align*}
A similar argument as the one used to check the well definition of \eqref{solmain} gives us that $\mathcal G$ is a well defined functional for every $u\in W^s_0L_M(\Omega)$.

We observe that $W_0^sE_M(\Omega)\subset \Dom(\A)\subset \Dom(\EE)\subset W_0^s L_M(\Omega)$, and the functions $\EE$ and $\G$ vanish only at zero. These inclusions are a direct consequence of the point wise inequality
$$
A(x, y, \xi)\le \xi a(x, y, \xi).
$$

Let $B=\{\phi_1,\phi_2,\dots\}\subset {\mathcal D}(\Omega)$ be a countable linearly independent subset in $W_0^sE_M(\Omega)$ such that the linear hull is norm dense, and we define the sets $V$ and $V_k$ as the linear hull of the sets $B$ and $B_k:=\{\phi_1,\phi_2,\dots,\phi_k\}$ respectively. We denote the continuous pairing between $W^{-s}L_{\bar{M}}(\Omega)$ and $W^s_0E_M(\Omega)$ by $\langle \cdot, \cdot \rangle$, and the one between $V_k$ and $(V_k)^*$ by $\langle \cdot, \cdot \rangle_k$.

A straightforward calculation gives
\begin{align*}
\langle (\EE)'(u),v \rangle_k&=\iint_{\R^{2n}} a(x,y,D^su)D^s v\,d\nu_n &\forall u,v\in V_k, 	
\\
\langle \G'(u),v\rangle_k&=\int_\Omega g(u)v\,dx &\forall u,v\in V_k, 	
\end{align*} 
for each $k=1,2,\dots,$ and the fact $V_k\subset W^s_0E_M(\Omega)$ allows us to conclude $\EE,\G \in C^1(V_k)$.

By the growth condition \eqref{condg} and the compact immersion $W^s_0L_M(\Omega) \subset\subset E_M(\Omega)$ (see Theorem B.3 in \cite{FBSP}),  it follows that $\Dom(\G)=W_0^sL_M(\Omega)$ and 
\begin{align*}
\G(u_k) &\to \G(u)\\
\langle \G'(u_k),u_k\rangle_k &\to \int_\Omega g(u) u\, dx\\
\langle \G'(u_k),v\rangle &\to \int_\Omega g(u) v\, dx,
\end{align*}
whenever $u_k\in V_k, u_k\to u\in W_0^sL_M(\Omega)$ in $\sigma(W_0^sL_M(\Omega),W^{-s}E_{\bar{M}}(\Omega)),$ and $v\in V$.

In order to analyze existence of solution, we reduce the problem \eqref{mainproblem} on the finite dimensional space $V_k$ that is
\begin{equation}\label{eigenfinito}
\begin{split}
& (\EE)'(u)=\lambda \mathcal{G}'(u)\text{ in }V_k^*
\\
&u\in \mathcal{N}_k^s,\lambda\in \R, 
\end{split}
\end{equation}
where $\mathcal{N}_k^s=\{u\in V_k\colon \EE(u)=1\}$.

Now, to attack the problem \eqref{mainproblem} we need the following notation 
\begin{align*}
\mathcal{N}^s	&=\{u\in W_0^sE_M(\Omega)\colon \EE(u)=1\}
\\
\mathcal{K}_i^s&=\{K\subset \mathcal{N}^s\text{ compact and symmetric}\colon \text{gen}(K)\geq i\}
\\
\mathcal{K}_{i,k}^s&=\{K\subset \mathcal{N}_k^s\text{ compact and symmetric}\colon \text{gen}(K)\geq i \}
\\
c_i&=\sup_{K\in \mathcal{K}_i^s}\inf_{u\in K}\G(u)
\\
c_{i,k}&=\sup_{K\in \mathcal{K}_{i,k}^s}\inf_{u\in K}\G(u),
\end{align*}
where $\text{gen}(K)$ stands for the Krasnoselskii genus (see \cite[Chapter 7]{Rabi}).

The proof now will make use of the finite dimensional Ljusternik-Schnirelman theory developed in \cite{Zei80}. The following lemma contains the necessary properties for the operator $\EE$ and the level set $\mathcal{N}_k^s$ needed in order to apply \cite[Theorem 2]{Zei80}.

\begin{lema}\label{PropEsa}
\begin{itemize}
\item[a)] $(\EE)'\colon V_k\to V_k^*$ is a potential operator with potential $\EE$.
\item[b)] $(\EE)'$ is uniformly continuous on bounded sets.
\item[c)] There exists a real number $r(u)>0$ such that $\EE(r(u)u)=1$ for every $u\neq 0$.
\item[d)] The level set $\mathcal{N}_k^s=\{u\in V_k\colon \EE(u)=1\}$ is bounded.
\item[e)] If $u\neq 0$ then  $\langle (\EE)'(u),u \rangle>0$ and  $\inf_{u\in \mathcal{N}_k^s}\langle (\EE)'(u),u \rangle>0$.  
\end{itemize}	
\end{lema}
\begin{proof}
The proof of a) has been observed before.

The item b) is consequence of the uniform continuity theorem.

We prove now item c). Let $u\in V_k$ with $u\not\equiv 0$. Define the function $\varphi\colon \R_+\to \R$ as
$$
\varphi(r):=\EE(ru)=\iint_{\R^{2n}}A(x, y, rD^su)\, d\nu_n.
$$	
We observe that $\varphi$ is a continuous function with $\varphi(0)=0$. Now, for $r>1$, we have
$$
\EE(ru)=\iint_{\R^{2n}} A(x,y,r |D^s u|)\, d\nu_n
$$
On the other hand, for $\xi>0$,
$$
A(x, y, r\xi) = \int_0^{r\xi} a(x, y, \tau)\, d\tau \ge \int_{\xi}^{r\xi} a(x, y, \tau)\, d\tau\ge (r-1)a(x, y, \xi)\xi.
$$
Therefore, we obtain
$$
\varphi(r)=\EE(ru)\ge (r-1)\iint_{\R^{2n}}a(x, y, |D^s u|)|D^s u|\, d\nu_n.
$$
Using now \eqref{MC} we obtain that $\varphi(r)\to\infty$ when $r\to\infty$. Then, by the intermediate value theorem, there exists $0<r(u)$ such that $\varphi(r(u))=1$, that is $\EE(r(u)u)=1$.

Item d) is consequence of item c). Indeed the map $h\colon S^{k-1} \to \mathcal{N}_k^s$, where $S^{k-1}=\{u\in V_k\colon \|u\|_{W^sL_M}=1\}$, defined as $h(u)=r(u)u$ is an homeomorphism therefore $\mathcal{N}_k^s$ is compact and in particular is bounded.

Finally to prove item e) we observe, by the monotonicity of $a(x,y,\cdot)$, 
\begin{equation}\label{cotaA}
A(x, y, \xi) = \int_0^{\xi} a(x, y, \tau)\, d\tau \le a(x, y, \xi)\xi
\end{equation}
and
$$
\langle (\EE)'(u),u \rangle=\iint_{\R^{2n}} a(x,y,D^su)D^s u\,d\nu_n >0 \quad \forall u\neq 0.
$$
Finally, if $u \in \mathcal{N}_k^s$ using \eqref{cotaA} we have
$$
1=\EE(u)=\iint_{\R^{2n}}A(x,y,D^s u)\, d\nu_n \leq \iint_{\R^{2n}}a(x,y,D^su)D^su\,d\nu_n,
$$
this implies $\inf_{u\in \mathcal{N}_k^s}\langle (\EE)'(u),u \rangle\ge 1>0$.  
\end{proof}

Now we are in position to provide solution for the problem \eqref{eigenfinito}.
\begin{lema}\label{soleigenfinito}
Let $i\in \N$ given. Then, for any $k\ge i$  there exist $u_k^i\in W_0^s E_M(\Omega)$, and $\lambda_k^i\in  (0,\infty)$ such that 
\begin{align*}
u_k^i&\in \mathcal{N}_k^s\subset V_k
\\
(\EE)'(u_k^i)&=\lambda_k^i \mathcal{G}'(u_k^i)\qquad \text{in }V_k^{*}
\\
\mathcal{G}(u_k^i)&=c_{i,k}.
\end{align*}
\end{lema}

\begin{proof}
The content of Lemma \ref{PropEsa} puts us in a position to directly apply \cite[Theorem 2 and Corollary 7.1]{Zei80} and that gives us exactly the claim of this lemma.
\end{proof}

In order to show analyze the behavior of the sequence $\{(\lambda_k^i, u_k^i)\}_{k\ge i}$ as $k\to\infty$ we need first to prove that the map $u\mapsto r(u)u$ given by Lemma \ref{PropEsa} defines an homeomorphism between the unit ball of $W^s_0E_M(\Omega)$ and the manifold $\mathcal{N}^s$. We begin with a lemma that shows the continuity of $r(u)$.
\begin{lema}\label{r.cont}
Let $r(u)$ be the unique positive real number $r$ such that $\EE(ru)=1$. Assume that $\|u_k\|_{s, M} = 1$ and that $u_k\to u$ in $W^s_0E_M(\Omega)$, then $r(u_k)\to r(u)$.
\end{lema}

\begin{proof}
We divide the proof into several steps for the reader convenience.

{\em Step 1.} We claim that $r_k:=r(u_k)$ is bounded. In fact, we may assume that $r_k>1$ and arguing as in Lemma \ref{PropEsa}, we get
\begin{align*}
1 = \EE(r_k u_k) \ge (r_k-1)\iint_{\R^{2n}} a(x, y, D^s u_k) D^s u_k\, d\nu_n.
\end{align*}
Next we use the coercivity condition \ref{coersive} to obtain
$$
1\ge (r_k-1) \theta \iint_{\R^{2n}} M(cD^s u_k)\, d\nu_n.
$$
Now we claim that there exists a constant $\delta>0$ such that $\iint_{\R^{2n}} M(cD^s u_k)\, d\nu_n\ge \delta$ for every $k\in\N$. Otherwise, $\iint_{\R^{2n}} M(cD^s u_k)\, d\nu_n\to 0$ up to a subsequence, but in turn this implies that $cu_k\to 0$ in $W^s_0E_M(\Omega)$ which is a contradiction.

Combining these estimates. we arrive at
$$
1\ge (r_k-1)\theta\delta
$$
and the claim follows.

Without loss of generality, we can assume that $r_k\to r_0$ as $k\to\infty$.

\medskip

{\em Step 2.} We claim that the functional $\EE$ is continuous in $W^s_0E_M(\Omega)$.

Using the monotonicity of $a(x, y, \xi)$ together with \eqref{GC} and Young's inequality \eqref{young}, we obtain the pointwise estimate
\begin{align*}
A(x, y, \xi) &\le \xi a(x, y, \xi) \le \xi (d(x, y) + b \bar M^{-1}(M(c\xi)))\\
&\le \bar M(d(x, y)) + M(\xi) + b(M(\xi) + M(c\xi)) \le \bar M(d(x, y)) + \tilde b M(c\xi).
\end{align*}

Observe now that if $v_k\to v$ in $W^s_0E_M(\Omega)$, then $M(cD^s v_k)\to M(c D^s v)$ $\nu_n-$a.e. and, moreover, $\iint_{\R^{2n}}M(cD^s v_k)\, d\nu_n\to  \iint_{\R^{2n}}M(cD^s v)\, d\nu_n$. Therefore, we can apply the Dominated Convergence Theorem to conclude that
$$
\iint_{\R^{2n}} A(x, y, D^s v_k)\, d\nu_n\to \iint_{\R^{2n}} A(x, y, D^s v)\, d\nu_n,
$$
from where the claim follows.

As a corollary of Steps 1 and 2, we have that $\EE(r_k u_k)\to \EE(r_0 u)$.

\medskip

{\em Step 3.} Since $\EE(r_k u_k)=1$ for every $k\in\N$, it follows from the previous steps that $\EE(r_0 u)= 1$. But, since $\varphi(r) = \EE(ru)$ is strictly increasing, we can conclude that $r(u)$ is unique, and so $r(u)=r_0$.
\end{proof}

\begin{lema}\label{invH.cont}
There exist constants $\delta>0$  such that $\delta \le \|u\|_{s, M}\le \delta^{-1}$ for every $u\in \mathcal N^s$. i.e. $\mathcal N^s$ is bounded away from zero and infinity.
\end{lema}

\begin{proof}
Observe that the proof of Lemma \ref{r.cont}, Step 1 gives us that $\|u\|_{s, M}$ is uniformly bounded for $u\in \mathcal N^s$. In fact, if $u\in \mathcal N^s$, then $v=u/\|u\|_{s, M}$ satisfies that $\|v\|_{s, M}=1$ and $r(v)=\|u\|_{s, M}$. Since Step 1 in Lemma \ref{r.cont} says that $r(v)$ is uniformly bounded for $\|v\|_{s, M}=1$ then the claim follows.

In order to obtain the lower bound, we argue similarly to the proof of Step 2 in Lemma \ref{r.cont}.

\end{proof}
\begin{prop}\label{H.homeo}
The map $H\colon B_{W^s_0E_M(\Omega)} \to \mathcal N^s$ given by $H(u) = r(u)u$ is an homeomorphism.
\end{prop}

\begin{proof}
We observe that Lemma \ref{r.cont} implies that $H$ is continuous. On the other hand a straightforward calculation gives $H^{-1}(v) = v/\|v\|_{s, M}$. Finally, Lemma \ref{invH.cont} give us the continuity of $H^{-1}$. Therefore $H$ is a homeomorphism.
\end{proof}

The following theorem gives the asymptotic behavior of the sequence $\{(\lambda_k^i, u_k^i)\}_{k\ge i}$.
\begin{teo}\label{teolambdaiui}
Let $i\in\N$ fixed and let $\{(\lambda_k^i, u_k^i)\}_{k\ge i}\subset (0, +\infty)\times W^s_0 E_M(\Omega)$ be the sequence given in Lemma \ref{soleigenfinito}. Then there exists $u_i\in \Dom(\A)$ and $\lambda_i\in (0,+\infty)$ such that up to a subsequence, $\lambda_k^i\to \lambda_i$ and $u_k^i\to u_i$ for $\sigma(W_0^sL_M(\Omega),W^{-s}E_{\bar{M}}(\Omega))$. Moreover, $\EE(u_i)=1,\G(u_i)=\lim_{k\to\infty}c_{i,k}$ and $(\lambda_i,u_i)$ is an eigenpair for \eqref{mainproblem}.	
\end{teo}
\begin{proof}
We observe that $\{u_k^i\}_{k\ge i}\subset \mathcal{N}_k^s\subset \mathcal{N}^s$ and by Lemma \ref{H.homeo} there exists an homeomorphism $H\colon \{u\in W_0^s E_M(\Omega)\colon \|u\|_{W^sL_M(\Omega)}=1\}\to \mathcal{N}^s$ defined as $H(u)=r(u)u$ where $r(u)>0$ is the same that in Lemma \ref{PropEsa}, therefore $\mathcal{N}^s$ is bounded and so $\{u_k\}_{k=i}^\infty$ is uniformly bounded  in $W_0^sL_M(\Omega)$ then there exists $u_i\in W_0^sL_M(\Omega)$ such that $u_k\to u_i$ in $\sigma(W_0^sL_M(\Omega),W^{-s}E_{\bar{M}}(\Omega))$, Lemma \ref{soleigenfinito} give us $\lim_{k\to \infty} c_{i,k}=\lim_{k\to\infty}\G(u_k)=\G(u_i)$.

We prove now the convergence of the sequence $\{\lambda_k\}_{k=i}^\infty$. Suppose that $\lambda_k\to +\infty$ if $k\to \infty$, since $\{c_{i,k}\}_{k=i}^{\infty}$ are increasing and positive we have $\G(u_i)>0$ implying $g(u_i)\not\equiv  0$ (and therefore $u_i\not\equiv 0$). Then there exists $\phi\in  V_{k_0}$ for some $k_0>i$ such that 
$$
\int_\Omega g(u_i)(u_i-\phi)\,dx<0.
$$    
Using the monotonicity condition  \eqref{MC} we have
$$
\iint_{\R^{2n}}(a(x,y,D^s u_k)-a(x,y,D^s \phi))(D^s u_k-D^s \phi)\,d\nu_n\geq 0,
$$
from where we obtain
$$
\iint_{\R^{2n}} a(x,y,D^s u_k)(D^s u_k-D^s \phi)\,d\nu_n\geq \iint_{\R^{2n}} a(x,y,D^s \phi )(D^s u_k-D^s \phi)\,d\nu_n, 
$$
and using that $(\lambda_k,u_k)$ is solution 
$$
\lambda_k\int_\Omega g(u_k)(u_k-\phi)\,dx\geq \iint_{\R^{2n}} a(x,y,D^s \phi )(D^s u_k-D^s \phi)\,d\nu_n,
$$
the left hand side in the above inequality tends to $-\infty$ and this is a contradiction therefore there exists $\lambda_i$ such that $\lambda_k\to \lambda_i$ if $k\to\infty$.

Now we prove that the sequence $\{a(x,y,D^su_k)\}_{k=i}^\infty$ is uniformly bounded in $L_{\bar{M}}(\nu_n)$. Let $V\in L^{\infty}(\R^{2n},\nu_n)$ with compact support, again by the monotonicity condition \ref{MC} we have
$$	
\iint_{\R^{2n}}(a(x,y,D^s u_k)-a(x,y,V))(D^s u_k-V)\,d\nu_n\geq 0,
$$
from where we obtain
\begin{align*}
\iint_{\R^{2n}}a(x,y,D^s u_k) V\,d\nu_n &\leq \iint_{\R^{2n}}a(x,y,D^s u_k) D^s u_k\,d\nu_n -\\
& - \iint_{\R^{2n}}a(x,y,V) D^s u_k\,d\nu_n+	
\\
& +\iint_{\R^{2n}}a(x,y,V) V\,d\nu_n
\\
& = I+II+III,
\end{align*}
the first integral $I$ is $\lambda_k \int_\Omega g(u_k) u_k dx\to \lambda_i \int_\Omega g(u_i) u_i dx$ therefore $I$ is bounded, the second integral $II$ is bounded since $u_k\to u_i$ in $\sigma(W_0^s L_M(\Omega),W^{-s}L_{\bar{M}}(\Omega))$ and third integral $III$ doesn't depends on $k$ then there exists a constant $C>0$ such that 
$$
\iint_{\R^{2n}}a(x,y,D^s u_k) V\,d\nu_n\leq C,
$$
for all $V\in L^{\infty}(\R^{2n},\nu_n)$ with compact support and by density for all $V\in E_M(\nu_n)$ therefore by the uniform boundedness principle the sequence $\{a(x,y,D^su_k)\}_{k=i}^\infty$ is uniformly bounded in $L_{\bar{M}}(\nu_n)$ and so there exists $F\in L_{\bar{M}}(\nu_n)$ such that $a(x,y,D^su_k)\to F$ in $\sigma(L_{\bar{M}}(\nu_n),E_M (\nu_n))$ and  $ \A u_k=-2\mathrm{div}^s(a(x,y,D^su_k))\to -2 \mathrm{div}^s(F):=f$ in $\sigma(W^{-s}L_{\bar{M}}(\Omega), W_0^s E_M(\Omega))$. Also, for $\phi\in {\mathcal D}(\Omega)$ we have
\begin{align*}
\langle f,\phi\rangle&=\lim_{k\to\infty} \langle \A u_k,\phi \rangle
\\
&=\lim_{k\to\infty}\iint_{\R^{2n}}a(x,y,D^s u_k)D^s\phi\,d\nu_n\\
&=\lim_{k\to\infty} \lambda_k \int_\Omega  g(u_k)\phi\,dx
\\
&=\lambda_i	\int_\Omega  g(u_i)\phi\,dx,
\end{align*}
since ${\mathcal D}(\Omega)$ is norm dense in $W_0^s E_M(\Omega)$ and $W_0^s E_M(\Omega)$ is dense in $W_0^s L_M(\Omega)$ in the topology $\sigma(L_M\times L_M (\nu_n), L_{\bar{M}}\times L_{\bar{M}} (\nu_n))$ it follows that 
$$
\langle f,\phi\rangle=\lambda_i	\int_\Omega  g(u_i)\phi\,dx\qquad \forall \phi\in W_0^s L_M(\Omega.
$$
On the other hand
\begin{align*}
\lim_{k\to\infty} \langle \A u_k,u_k \rangle	&=\lim_{k\to\infty}\iint_{\R^{2n}}a(x,y,D^su_k)D^su_k\,d\nu_n
\\
&=\lim_{k\to\infty}\lambda_k \int_\Omega g(u_k)u_k\,dx
\\
&=\lambda_i\int_\Omega g(u_i)u_i\,dx
\\
&=\langle f,u_i \rangle.
\end{align*}
That is
\begin{equation}\label{convoperadores}
\lim_{k\to\infty} \langle \A u_k,u_k \rangle	= \langle f,u_i \rangle.
\end{equation}
The above and the pseudomonotonicity of the operator $\A$ implies that $u_i\in Dom(\A)$ and $\A u_i=f$, it follows for $\phi\in W_0^s L_M(\Omega)$
$$
\iint_{\R^{2n}}a(x,y,D^s u_i)D^s \phi\,d\nu_n=\langle \A u_i,\phi \rangle=\langle f,\phi\rangle=\lambda_i	\int_\Omega  g(u_i)\phi\,dx,
$$ 
then $(\lambda_i,u_i)\in (0,+\infty)\times Dom(\A)$ is an eigenpair for \eqref{mainproblem}.

Finally, by Lemma 4.4 in \cite{FBSP} and \eqref{convoperadores} we have that $a(x,y,D^s u_k)D^s u_k\to a(x,y,D^s u_i)D^s u_i$ in $L^1(\R^{2n},d\nu_n)$ and so, by \cite[Theorem 4.9]{brezis}, there exists a integrable majorant $h\in L^1(\R^{2n},d\nu_n)$ such that $|a(x,y,D^s u_k)||D^s u_k|\leq h$ $\nu_n-$a.e. in $\R^{2n}$  then 
$$
|A(x,y,D^s u_k)|=\left | \int_0^{D^s u_k}a(x,y,\tau)\,d\tau\right |\leq a(x,y,D^s u_k)D^s u_k\leq h,
$$
$\nu_n-$a.e. in $\R^{2n}$. By the compact immersion $W^s_0 L_M(\Omega)\subset \subset E_M(\Omega)$ there exists a subsequence that we still calling $\{u_k\}_{k=i}^\infty$ such that $u_k\to u_i$ a.e. in $\Omega$, and extending $u_i=0$ in $\Omega^c$ we have $D^su_k\to D^s u_i$ $\nu_n-$a.e. in $\R^{2n}$, and consequently $A(x,y,D^s u_k)\to A(x,y,D^s u_i)$ $\nu_n-$a.e. in $\R^{2n}$. It follows from the dominated convergence theorem
\begin{align*}
\EE(u_i)&=\iint_{\R^{2n}} A(x,y,D^s u_i)\,d\nu_n	
\\
&=\lim_{k\to\infty}\iint_{\R^{2n}}A(x,y,D^s u_k)\,d\nu_n
\\
&=\lim_{k\to\infty}\EE(u_k)
\\
&=1.
\end{align*}
\end{proof}
Now to give the correct asymptotic behavior of the eigenpairs $\{(\lambda_i,u_i)\}_{i=1}^\infty$ we need the following two lemmas whose proofs can be found in \cite{FBSP} or \cite{Ti}.
\begin{lema}
	\label{ciktoci}
Let $i\in\N$ be fixed. Then $c_{i,k}\to c_i$ as $k\to\infty$.
\end{lema}
\begin{lema}\label{cito0}
$c_i\to 0$ as $i\to \infty$.
\end{lema}
We can give the proof of the main result.
\begin{teo}\label{teo.main}
Let $\Omega\subset \R$ be an open and bounded domain which satisfies the segment property. Then there exists  a sequence of eigenpairs $\{(\lambda_i,u_i)\}_{i=1}^\infty\subset \R_+\times W_0^sL_M(\Omega)$  of \eqref{mainproblem}. Moreover $\lambda_i\to+\infty$ and $u_i\to 0$ in $\sigma(W_0^sL_M(\Omega), W^{-s}E_{\bar{M}}(\Omega))$ as $i\to \infty$. 	
\end{teo}
\begin{proof}
Theorem \ref{teolambdaiui}, Lemma \ref{ciktoci} and Lemma \ref{cito0} imply that there exists a sequence of eigenpairs $\{(\lambda_i,u_i)\}_{i=1}^\infty\subset \R_+\times W_0^sL_M(\Omega)$ for the problem \eqref{mainproblem} and $\G(u_i)\to 0$ if $i\to\infty$. Using Theorem \ref{teolambdaiui} and \eqref{coersive} we have
\begin{align*}
1=&\EE(u_i)
\\
=&\iint_{\R^{2n}}\int_0^{D^s u_i}a(x,y,\tau)\,d\tau\,d\nu_n
\\
\geq & \iint_{\R^{2n}}a \left (x,y,\frac{D^su_i}{2}\right)\frac{D^su_i}{2}\,d\nu_n	
\\
\geq & \iint_{\R^{2n}}\theta M\left (c \frac{D^s u_i}{2}\right )\,d\nu_n, 
\end{align*}
where $\theta,c>0$ are constants independent of $i$. Therefore we have
$$
\iint_{\R^{2n}} M\left (c \frac{D^s u_i}{2}\right )\,d\nu_n\leq \frac{1}{\theta},
$$
from where we obtain
$$
\|D^su_i\|_{M,\nu_n}\leq \frac{2 \max\{1, 1/\theta\}}{c},
$$
and so the sequence $\{u_i\}_{i=1}^\infty$ is uniformly bounded in $W_0^s L_M(\Omega)$ then there exists $\tilde{u}\in W_0^sL_M(\Omega)$ such that $u_i\to \tilde{u}$ in $\sigma(W_0^sL_M(\Omega), W^{-s}E_{\bar{M}}(\Omega))$ and $\G(u_i)\to \G(\tilde{u})$ as $i\to\infty$. From the above we have $\tilde{u}=0$ and therefore $u_i\to 0$ in $\sigma(W_0^sL_M(\Omega), W^{-s}E_{\bar{M}}(\Omega))$. 	

Finally, using the fact that for each $i$ the pair $(\lambda_i,u_i)$ is solution for \eqref{mainproblem}, the monotonicity condition \eqref{MC} of $a$ and Theorem \ref{teolambdaiui} we conclude
$$
\lambda_i=\frac{\iint_{\R^{2n}}a(x,y,D^su_i)D^su_i\,d\nu_n}{\int_\Omega g(u_i)u_i\,dx}\geq \frac{\EE (u_i)}{\int_\Omega g(u_i)u_i\,dx}= \frac{1}{\int_\Omega g(u_i)u_i\,dx},
$$
then $\lambda_i\to+\infty$ as $i\to\infty$.
\end{proof}

\section*{Acknowledgements}
This work benefited from the visit of the first author to the National University of San Luis and from the visit of the second author to the University of Buenos Aires. The authors are grateful to both institutions for their hospitality.

The first and third authors are members of CONICET, while the second author is a doctoral fellow of CONICET.

Unfortunately, research funding in Argentina has undergone significant cuts by the national government, and therefore no financial support was available for this work.

\bibliographystyle{amsplain}
\bibliography{biblio}
\end{document}